%% file: article.tex
\documentclass[11pt]{amsart} 
\usepackage{amssymb,amsmath,latexsym,enumerate,graphicx,bbm,mathptmx,microtype,cite,color}
\allowdisplaybreaks

\newtheorem{theorem}{Theorem} [section] \numberwithin{equation}{section}
\newtheorem{corollary}[theorem]{Corollary}

\newtheorem{lemma}[theorem]{Lemma}

\newtheorem{exam}{Example}

\newtheorem{rem}[theorem]{Remark}
\newenvironment{remark}{\begin{rem}\rm}{\end{rem}}

\renewcommand\emptyset{\varnothing}

\newcommand\commentout[1]{}
\newcommand\Def[1]{{\bf #1}}

\newcommand\vol{\operatorname{vol}} 
\newcommand\conv{\operatorname{conv}}

\newcommand\ehr{\operatorname{ehr}}

\newcommand\des{\operatorname{des}} 
\newcommand\aff{\operatorname{aff}} 

\newcommand\ZZ{\mathbb{Z}}

\newcommand\RR{\mathbb{R}}

\newcommand\bu{\mathbf{u}}
\newcommand\bv{\mathbf{v}}
\newcommand\bw{\mathbf{w}}
\newcommand\bx{\mathbf{x}}

\newcommand\bz{\mathbf{z}}
\newcommand\bzero{\mathbf{0}}

\newcommand\hstar{h^\ast}

\makeatletter 
\newtheorem*{rep@theorem}{\rep@title}\newcommand{\newreptheorem}[2]{%
\newenvironment{rep#1}[1]{%
\def\rep@title{\bf #2 \ref{##1}}%
\begin{rep@theorem}}%
{\end{rep@theorem}}}
\makeatother
\newreptheorem{theorem}{Theorem}

\begin{document}

\title{Lattice zonotopes of degree 2}

\author{Matthias Beck}
\address{Department of Mathematics\\
         San Francisco State University\\
         San Francisco, CA 94132\\
         U.S.A.}
\email{mattbeck@sfsu.edu}

\author{Ellinor Janssen}
\address{Mathematisches Institut, Freie Universit\"at, 14195 Berlin, Germany}
\email{ellinorjanssen@gmail.com}

\author{Katharina Jochemko}
\address{Department of Mathematics\\
         KTH Royal Institute of Technology\\
         100 44 Stockholm\\
         Sweden}
\email{jochemko@kth.se}



\date{20 September 2022}

\maketitle

\begin{abstract}
The \emph{Ehrhart polynomial} $\ehr _P (n)$ of a lattice polytope $P$ gives the number of
integer lattice points in the $n$-th dilate of $P$ for all integers $n\geq 0$. The
\emph{degree} of $P$ is defined as the degree of its $h^\ast$-polynomial, a
particular transformation of the Ehrhart polynomial with many useful properties which serves as an important tool for classification questions in Ehrhart theory.

A \emph{zonotope} is the Minkowski (pointwise) sum of line segments. We classify all
Ehrhart polynomials of lattice zonotopes of degree $2$ thereby complementing results
of Scott (1976), Treutlein (2010), and Henk--Tagami (2009). Our proof is constructive: by considering solid-angles and the lattice width, we provide a characterization of all $3$-dimensional zonotopes of degree~$2$.

\end{abstract}


\section{Introduction}

A \textbf{lattice polytope} $P \subset \mathbb{R}^d$ is the convex hull of finitely many points in
$\mathbb{Z}^d$. A fundamental result by Ehrhart states that the
number of lattice points in the $n$-th dilate of $P$ is given by a polynomial in the integer dilation factor~$n$. 
\begin{theorem}[Ehrhart~\cite{ehrhartpolynomial}]
Let $P\subset \mathbb{R}^d$ be a lattice polytope. Then there exists a polynomial $\ehr _P (n)$ such that $|nP\cap \mathbb{Z}^d|=\ehr _P (n)$ for all integers $n\geq 1$.
\end{theorem}
The polynomial $\ehr_P (n)$, which is referred to as the \textbf{Ehrhart polynomial} of $P$,
encodes fundamental properties of the polytope; for example, its degree equals the
dimension of $P$
and its leading coefficient is equal to the volume of~$P$.  As with any class of (combinatorial) polynomials, it is natural to ask how they might be classified.
Scott completely classified all Ehrhart polynomials of $2$-dimensional polytopes, as follows.

\begin{theorem}[Scott~\cite{scott}]\label{thm:scott}
The polynomial $1+e_1n+e_2n^2$ is equal to the Ehrhart polynomial $\ehr _P (n)$ of a $2$-dimensional lattice polytope $P$ if and only if $e_1, e_2 \in \frac 1 2 \mathbb{Z}_{>0}$ 
and one of the following three conditions is satisfied.
\begin{itemize}
\item[(i)] $e_2=e_1-1$, 
\item[(ii)] $e_1< e_2+1$ and $\tfrac{3}{2} \leq e_1 \leq \tfrac{1}{2}e_2+2$, or
\item[(iii)] $(e_1,e_2)=(\tfrac{9}{2},\tfrac{9}{2})$.
\end{itemize}
\end{theorem}

We note that (i) is equivalent to $P$ having no interior lattice points, 
and (iii) is satisfied if and only if $P$ is unimodularly equivalent to the triangle
$\Delta =\conv \{(0,0), (3,0),(0,3)\}$, where \Def{unimodular equivalence} in $\RR^d$, denoted by $\cong$, refers to equality up to a transformation in GL$_d (\ZZ)$ followed by a lattice translation.

In higher dimensions, the problem of characterization and classification of Ehrhart polynomials is wide open. In particular, in dimension $\geq 3$ the coefficients of the Ehrhart polynomial can be
negative in general (see, e.g.,~\cite{ccd}). An important tool to study
Ehrhart polynomials that remedies the issue of negativity is the \textbf{$h^\ast$-polynomial} $h^\ast _P (t)=\sum _{i=0}^d h_i^\ast t^i$ whose coefficients encode the Ehrhart polynomial in terms of binomial coefficients; if $\dim P=d$ then
\[
\ehr_P (n) \ = \ h^\ast_0 \binom{ n+d } d + h^\ast_1 \binom{ n+d-1 } d + \dots + h^\ast_d \binom n d \, .
\]
The coefficient $h^\ast _d$ is equal to the number of interior lattice points of $P$.
We note that $\binom{ n+d-i } d$ is a polynomial in $n$ for all $0\leq i \leq n$ and that $\left\{ \binom{ n+d } d ,\binom{ n+d-1 } d,\ldots,\binom n d\right\}$ is a basis for all polynomials of degree $\leq d$.
Stanley~\cite{stanleydecomp} proved that the coefficients $\hstar_j$ are always nonnegative integers. 

The \textbf{degree} of the lattice polytope $P$, denoted $\deg P$, is
defined to be the degree of $h^\ast _P (t)$. 
Since $\ehr _P (n)$ is a polynomial of degree $d$, the degree of $P$ is $\le d$. 
It is the smallest natural number $r$ such that the $(d-r)$-th dilate of $P$ does not contain interior lattice points. 

It is a short step to ask the classification question for Ehrhart polynomials for lattice
polytopes with certain degrees. For degree $2$ this was answered by Treutlein, who gave constraints on the coefficients of the $h^\ast$-polynomial of degree-$2$ lattice polytopes, independent of the dimension.

\begin{theorem}[Treutlein~\cite{treutlein}]\label{thm:treutlein}
Let $1 + \hstar_1 \, t + \hstar_2 \, t^2$ be the $\hstar$-polynomial of a lattice polytope. Then
\begin{itemize}
\item[(i)] $h_2^\ast=0$ or
\item[(ii)] $0\leq h_1^\ast\leq  3 \hstar_2 + 3$ or
\item[(iii)] $h_1^\ast =7$ and $h_2^\ast =1$\, .
\end{itemize}
\end{theorem}

In dimension $2$, the inequalities given in Theorem~\ref{thm:treutlein} together with the constraint $h_2^\ast\leq h_1^\ast$ completely classify all $h^\ast$-polynomials of degree $2$. In dimension $3$, Henk and Tagami~\cite{henktagami} showed that the conditions given in Theorem~\ref{thm:treutlein} are also sufficient and thereby characterized all Ehrhart polynomials of $3$-dimensional degree-$2$ polytopes.

The goal of this article is to derive classification results analogous to Theorems~\ref{thm:scott} and~\ref{thm:treutlein}
for the Ehrhart polynomials of degree-$2$ zonotopes. 

A \textbf{zonotope} is defined as the Minkowski sum of finitely many line segments. Given
vectors $\mathbf{v}_1,\ldots, \mathbf{v}_m \in \mathbb{R}^d$, the \textbf{zonotope} generated by $\mathbf{v}_1,\ldots, \mathbf{v}_m$ is defined by
\[
Z(\mathbf{v}_1,\ldots, \mathbf{v}_m) \ = \ \left\{\sum _{i=1}^m \lambda_i \mathbf{v}_i \colon 0\leq \lambda _i \leq 1 \text{ for all } i \right\} \, .
\]
Up to translation, every zonotope is of this form. A zonotope is called a \textbf{lattice zonotope} if it is a lattice polytope. This is the case if and only if it is a translate of a zonotope generated by vectors in $\mathbb{Z}^d$. Zonotopes form a fundamental class of polytopes, and $2$-dimensional zonotopes are precisely
the centrally symmetric polygons (see, e.g., \cite[Example 7.14]{ziegler}). The following expression for the Ehrhart polynomial of lattice zonotopes was given by
Stanley, following from a canonical subdivision of a zonotope into parallelepipeds~\cite{shephardzonotopes}.

\begin{theorem}[Stanley~\cite{stanleyzonotopegraphicaldegree}]\label{thm:Stanley}
Let $\mathbf{v}_1,\ldots, \mathbf{v}_m$ be integer vectors and let $Z := Z(\mathbf{v}_1,\ldots, \mathbf{v}_m)$. Then
\[
\ehr _Z (n) \ = \ \sum _I g(I)n^{|I|}
\]
where $I$ ranges over all linearly independent subsets $\{ \mathbf{v}_1,\ldots, \mathbf{v}_m
\} $, and $g(I)$ denotes the greatest common divisor of all minors of size $|I|$ of the
matrix with vectors indexed by elements of $I$ as columns, where $g(\emptyset):=1$.
\end{theorem}
In particular, it follows that Ehrhart polynomials of lattice zonotopes have only nonnegative integer coefficients. Further, their Ehrhart polynomials are constrained to the
following form. 

\begin{theorem}[Beck--Jochemko--McCullough~\cite{beulerian}]\label{cor:Ehrhartzonotope}
For all $d$-dimensional lattice zonotopes $P$ there exist (uniquely determined) nonnegative integers $c_1,\ldots, c_d$ such that
\[
\ehr _P (n) \ = \ (n+1)^d + c_1 (n+1)^{d-1} \, n + c_2 (n+1)^{d-2} \, n^2+\dots +c_d \, n^d \, .
\]
\end{theorem}

We again observe that the set of polynomials $\{(n+1)^d, (n+1)^{d-1}n,\ldots, n^d\}$ forms a basis for the vector space of polynomials of degree at most $d$. A useful feature of this basis is that the number of interior points of the polytope can be read off easily: by Ehrhart-Macdonald reciprocity (see, e.g.,~\cite{ccd}), this number is equal to $|\ehr_P(-1)|$, which equals the coefficient $c_d$ in this basis. In particular, $c_d=h^\ast _d$.

As we shall prove, the coefficients of Ehrhart polynomials of lattice zonotopes of degree $2$ with respect to this basis can be classified in a particularly simple form. Our main results are as follows.

\begin{theorem}\label{thm:2imzono}
The polynomial
\[
(1+n)^2+ c_1(n+1)n+c_2\, n^2
\]
is the Ehrhart polynomial of a $2$-dimensional lattice zonotope if and only if
\begin{enumerate}[{\rm (i)}]
\item $c_1,c_2\in \mathbb{Z}_{\geq 0}$ and
\item $c_2\geq c_1 -1$ or $c_2=0$.
\end{enumerate}
\end{theorem}
From Theorem~\ref{thm:2imzono} we furthermore obtain in Corollary~\ref{cor:scottzonotopes} that the coefficient vectors of $2$-dimensional lattice zonotopes correspond exactly to the integer solutions of the constraints given in Theorem~\ref{thm:scott}.
 
\begin{theorem}\label{thm:degree2zono}
Any lattice zonotope of degree $2$ has dimension $\le 3$.
The polynomial
\[
(1+n)^3+ c_1(n+1)^2n+c_2(n+1)n^2+c_3\, n^3
\]
is the Ehrhart polynomial of a $3$-dimensional lattice zonotope of degree $2$ if and only if
\begin{enumerate}[{\rm (i)}]
\item $c_1,c_2\in \mathbb{Z}_{\geq 0}$ and $c_3=0$ and
\item $c_2\geq c_1 -1$ or $c_2=0$.
\end{enumerate}
\end{theorem}
In particular, we observe that the classification results for Ehrhart polynomials of lattice zonotopes of degree $2$ given in Theorems~\ref{thm:2imzono} and~\ref{thm:degree2zono} are identical, up to a projection of the latter onto the first three coefficients. 

Our proof of Theorem~\ref{thm:degree2zono} is constructive:  we provide a characterization of the actual $3$-dimensional zonotopes of degree~$2$.
\begin{theorem}\label{thm:degree2zonoclass}
A $3$-dimensional lattice zonotope $P$ has degree $2$ if and only if either
\begin{enumerate}[{\rm (i)}]
\item $P \cong Q \times [0,1]$ for some $2$-dimensional lattice zonotope $Q$, or
\item $P\cong Z(\mathbf{v}_1,\mathbf{v}_2,\mathbf{v}_3)$ where
\[
\mathbf{v}_1=[1,1,0] \, , \ \mathbf{v}_2 = [-1,1,0] \, , \ \mathbf{v}_3 = [1,1,2] \, .
\]
\end{enumerate}
\end{theorem}
To prove our results we consider solid angles and the lattice width for zonotopes. In
particular, Theorem~\ref{thm:degree2zonoclass} shows that, apart from one exception, all
$3$-dimensional lattice zonotopes of degree $2$ have lattice width $1$. This complements
previous results on $3$-dimensional lattice polytopes of lattice width greater than $1$
and $3$-dimensional lattice polytopes without interior points; see, e.g., \cite{AverkovWagnerWeismantel,NillZiegler,BlancoSantos}.

The outline is as follows: in Section~\ref{sec:dim2} we consider $2$-dimensional lattice
zonotopes and prove Theorem~\ref{thm:2imzono}. As a corollary we obtain that the
coefficient vectors of all $2$-dimensional lattice zonotopes are exactly given by all
integer vectors satisfying the inequalities given in Theorem~\ref{thm:scott}.
Section~\ref{sec:dim3} is dedicated to the proofs of Theorems~\ref{thm:degree2zono} and~\ref{thm:degree2zonoclass}. We conclude in Section~\ref{sec:bases} by translating our classification results in terms of the $h^\ast$-polynomial. 

Throughout we assume basic knowledge of lattice polytopes and Ehrhart polynomials, and introduce essential concepts when they appear. For further reading, see, e.g.,~\cite{ccd}.

\section{Two-dimensional lattice zonotopes}\label{sec:dim2}
In this section we prove Theorem~\ref{thm:2imzono} thereby classifying all Ehrhart polynomials of $2$-dimensional lattice zonotopes.

\begin{proof}[Proof of Theorem~\ref{thm:2imzono}]
Let $(1+n)^2+ c_1n(n+1)+c_2\,n^2$ be the Ehrhart polynomial of a $2$-dimensional lattice zonotope $P$. Then the coefficients $c_1$ and $c_2$ are nonnegative integers by Theorem~\ref{cor:Ehrhartzonotope}. In
order to prove that condition (ii) is also necessary, we assume that $c_2 > 0$, that is, $P$ contains
interior lattice points. Since
\[
(n+1)^2+c_1 (n+1)n +c_2 \, n^2 \ = \ 1+(2+c_1)n+(c_1+c_2+1)n^2 \, ,
\]
Theorem~\ref{thm:scott} (ii) gives $2+c_1 \leq \tfrac{1}{2} (c_1+c_2+1)+2$ which is equivalent to $c_2\geq c_1 -1$. (Properties (i) and (iii) in Theorem~\ref{thm:scott} do not apply here since $P$ contains interior lattice points and is not a triangle.)

On the other hand, for all pairs $(c_1,c_2)$ satisfying the conditions (i) and~(ii) we
provide a lattice zonotope with Ehrhart polynomial $(1+n)^2+ c_1n(n+1)+c_2\,n^2$; if $c_1\in \mathbb{Z}_{\geq 0}$ and $c_2=0$ then the axes-parallel parallelogram $Z((1,0),(0,c_{1}+1))$ with side lengths $1$ and $c_1 + 1$ has Ehrhart polynomial 
\[
(n+1)((c_1+1)n+1) \ = \ (n+1)^{2}+c_{1}(n+1)n\,; 
\]
if $c_1,c_2\in \mathbb{Z}_{\geq 0}$ and $c_2\geq c_1 -1$ then, using Theorem~\ref{thm:Stanley}, we can calculate that the Ehrhart polynomial of the lattice zonotope
$Z((1,0), (0,c_{1}), (1,1-c_{1}+c_{2}))$ is 
\[
(n+1)^2+c_1 (n+1)n +c_2 \, n^2. 
\]
This proves the sufficiency of the conditions (i) and (ii).
\end{proof}

\begin{corollary}\label{cor:scottzonotopes}
The set of Ehrhart polynomials of $2$-dimensional lattice zonotopes is equal to the set of polynomials described in Theorem~\ref{thm:scott} adding the condition that $(e_{1},e_{2}) \in \mathbb{Z}_{>0}\times \mathbb{Z}_{>0}$.
\end{corollary}

\begin{proof}
Let $1+e_1n+e_2n^2$ be the Ehrhart polynomial of a $2$-dimensional lattice zonotope. Then
the pair of coefficients $(e_1,e_2)$ satisfies the conditions of Theorem~\ref{thm:scott},
and $e_1$ and $e_2$ are integers by Theorem~\ref{thm:Stanley}. To see that every such pair
of integers arises as coefficients of a lattice zonotope we observe that the function
$\mathbb{Z}_{\geq0}\times \mathbb{Z}_{\geq0} \rightarrow \mathbb{Z}_{>0}\times
\mathbb{Z}_{>0}$ given by 
\[
(c_{1},c_{2}) \ \mapsto \ (e_{1},e_{2}):=(2+c_{1},1+c_{1}+c_{2})
\]
maps the coefficients with respect to the basis $\{(n+1)^{2-j}n^j\}_{j=0}^2$ to the coefficients with respect to the standard basis. Further, $c_{2}=0$ holds if and only if $e_{2}=1+c_{1}=(2+c_{1})-1=e_{1}-1$, that is, the second alternative of condition (ii) in Theorem~\ref{thm:2imzono} and condition (i) in Theorem~\ref{thm:scott} are equivalent. By the same argument, $c_2>0$ is satisfied if and only if $e_2>e_1-1$, and 
\[
 \tfrac{3}{2}  \ \leq \ e_1 \ = \ 2+c_{1}
 \ \leq \ 2+\tfrac{1}{2}c_{1}+\tfrac{1}{2}(1+c_{2})
 \ = \ 2+\tfrac{1}{2}(1+c_{1}+c_{2})
 \ = \ 2 +\tfrac{1}{2}e_2 
\]
holds if and only if $c_1\leq 1+c_2$. Thus, the first alternative of condition (ii) in
Theorem~\ref{thm:2imzono} is equivalent to condition (ii) in Theorem~\ref{thm:scott}. This proves the claim.

\end{proof}

\section{Zonotopes of degree two}\label{sec:dim3}

By a result of Shephard~\cite{shephardzonotopes}, any lattice zonotope can be subdivided into lattice parallelepipeds
which, in turn, are easily seen to contain interior integer points in their second dilates: the second dilate of $Z(\mathbf{v}_1,\ldots, \mathbf{v}_m)$ contains the sum of all of its generators in the interior, and similarly for integer translates of these parallelepipeds. This means the
degree of a lattice zonotope of dimension $d$ is either $d-1$ or $d$. Thus, zonotopes of degree $2$
can be of dimension $2$---which was covered by the previous section---or $3$, which is the setting
of this section. Our goal is to prove Theorems~\ref{thm:degree2zono} and~\ref{thm:degree2zonoclass} thereby providing a characterization of all $3$-dimensional lattice zonotopes of degree $2$ which will then lead to a complete classification of their Ehrhart polynomials.

A useful tool for our considerations are solid angles. Given a $d$-dimensional polytope $P\subset \mathbb{R}^d$, the \textbf{solid angle} at a point $\mathbf{x}\in \mathbb{R}^d$ with respect to $P$ is defined as
\[
\alpha _P (\mathbf{x}) \ := \ \lim _{\varepsilon \rightarrow 0} \frac{\vol (B_\varepsilon (\mathbf{x})\cap P)}{\vol (B_\varepsilon (\mathbf{0}))}
\]
where $B_\varepsilon (\mathbf{x})$ denotes the Euclidean ball around $\mathbf{x}$ with radius $\varepsilon$. The \textbf{solid-angle sum} of a lattice polytope $P$ is defined as
\[
A (P) \ := \sum _{\mathbf{x}\in P\cap \mathbb{Z}^d}\alpha _P(\mathbf{x}) \, .
\] 
We will also consider solid angles of lower-dimensional lattice polytopes with respect to the induced norm and volume in their affine hull.

The following result was obtained in~\cite{gravinrobinsshiryaev} (see Theorems~1.2 and 6.1 in~\cite{gravinrobinsshiryaev} and use the fact that $\vol Z$ is the multiplicity of the
multi-tiling by translates of $Z$), more generally, for lattice polytopes that multi-tile $\RR^d$.
\begin{lemma}[{Gravin--Robins--Shiryaev~\cite{gravinrobinsshiryaev}}]\label{lem:solidangleconstant}
Let $Z$ be a $d$-dimensional lattice zonotope in $\mathbb{R}^d$. Then $A(Z+\mathbf{t})=\vol Z$ for all $\mathbf{t}\in \RR^d$. In particular, every translate of $Z$ contains a lattice point.
\end{lemma}

The following lemma for $2$-dimensional lattice zonotopes will be useful further below. It follows from \cite[Corollary 3.6]{codenottisantosschymura}, which more
generally proves that every lattice polygon with an interior lattice
point has a covering radius strictly less than 1, except for three polygons which are all not
zonotopes that we consider here. We provide a different (and arguably more elementary) proof here.

\begin{lemma}\label{2dimintpoint}
Every translate of a $2$-dimensional lattice zonotope in $\mathbb{R}^2$ generated by three or more pairwise linearly independent vectors contains an interior lattice point.
\end{lemma}
\begin{proof}
It suffices to show that for any lattice zonotope $Z = Z (\bv_{1}, \bv_{2}, \bv_{3})$ generated by three pairwise linearly independent vectors $\bv_{1}, \bv_{2}, \bv_{3} \in \mathbb{Z}^{2}$, and any translation vector $\mathbf{t}\in \mathbb{R}^2$ the translated zonotope $Z+\mathbf{t}$ contains an interior lattice point.

Since $\bv_{1}, \bv_{2}, \bv_{3}$ are linearly dependent there exist $\alpha_{2}, \alpha_{3} \in \mathbb{R}_{> 0}$ and $\sigma_{2}, \sigma_{3}
\in \{ \pm 1\}$ such that $\bv_{1} = \sigma_{2} \alpha_{2} \bv_{2} + \sigma_{3} \alpha_{3} \bv_{3}$. Without loss of generality we may assume $\sigma _2,\sigma _3=1$ since $Z(\pm \mathbf{v}_1,\pm \mathbf{v}_2,\pm \mathbf{v}_3)$ (with exactly one fixed sign per
generator) is an integer translate of $Z$.
Let
$$\epsilon \ := \ \frac{1}{2 \max \{1, \alpha_{2}, \alpha_{3}\}} \, .$$
Then $\epsilon \in (0,\tfrac{1}{2})$ and $\epsilon \alpha_{i} \in (0,\tfrac{1}{2})$ for $i = 2,3$. The set $ V := (1 - \epsilon) \bv_{1} + Z (\bv_{2}, \bv_{3})$ is a subset of the zonotope $Z$ and a translate of the lattice parallelepiped $Z (\bv_{2}, \bv_{3})$. By Lemma \ref{lem:solidangleconstant}, the translate $\mathbf{t} + V$ contains an
integer point $\bx$. We argue by case distinction.

If $\bx$ lies in the interior of $\mathbf{t} + V$, it is also an interior lattice point of $\mathbf{t} + Z$ and the claim follows.

If $\bx$ is a vertex of $\mathbf{t} + V$, all vertices of $\mathbf{t} + V$ are lattice points and so, in
particular, $(1 - \epsilon) \bv_{1} + \mathbf{t} \in \ZZ^2$.
Since we can write
$$(1 - \epsilon) \bv_{1} + \mathbf{t} \ = \ \left( \left(1 - \tfrac{3}{2} \epsilon \right) \bv_{1} +
\tfrac{\epsilon \alpha_{2}}{2} \bv_{2} + \tfrac{\epsilon \alpha_{3}}{2} \bv_{3} \right) + \mathbf{t}
\, ,$$
where $0< 1 - \tfrac{3}{2} \epsilon,\tfrac{\epsilon \alpha_{2}}{2}, \tfrac{\epsilon \alpha_{3}}{2}< 1$
we obtain that $(1 - \epsilon) \bv_{1} +\mathbf{t}$ lies in the interior of $\mathbf{t} + Z$.

It remains to consider the case in which $\bx$ is in the relative interior of one of the facets of $\mathbf{t} + V$. Then there exists
another integer point $\bar{\bx}$ in the relative interior of the opposing facet.
Without loss of generality we may assume that $\bx = \left( (1 - \epsilon) \bv_{1} + \lambda \bv_{2} + 0 \cdot \bv_{3} \right) + \mathbf{t}$ for some $\lambda \in (0,1)$. Then for any $0<\delta<\min \left\{1 - \epsilon, \tfrac{1 - \lambda}{\alpha_{2}}, \tfrac{1 }{\alpha_{3}}\right\} $ we can write
$$\bx
  \ = \ \left( (1 - \epsilon) \bv_{1} + \lambda \bv_{2} + 0 \cdot \bv_{3} \right) + \mathbf{t}
  \ = \ \left( (1 - \epsilon - \delta) \bv_{1} + (\lambda + \delta \alpha_{2}) \bv_{2} + \delta \alpha_{3} \bv_{3} \right) + \mathbf{t}\, ,$$
  where $0<(1 - \epsilon - \delta),\lambda + \delta \alpha_{2},\delta \alpha_{3}<1$. Thus, again we see that $\bx$ is an interior lattice point of $\mathbf{t}+Z$. This completes the proof.
\end{proof}
We observe that Lemma~\ref{2dimintpoint} also holds for translates of $2$-dimensional lattice zonotopes within their affine hulls in higher dimensions.

For every $\bv \in \ZZ^3\setminus \{ \bzero \}$, let $\mathcal{L} (\bv)= |Z (\bv) \cap
\mathbb{Z}^{3}|-1$ denote the relative volume of the line segment $Z
(\bv) = \{\lambda \bv: \lambda \in [0,1]\}$ with respect to the lattice structure of the line
$\{\lambda \bv: \lambda \in \mathbb{R}\}$. We call $\bv$ \textbf{primitive} if $\mathcal{L} (\bv) = 1.$ The \textbf{width} of a lattice polytope $P$ in direction $\mathbf{v}\in
\mathbb{R}^d\setminus \{ \bzero \}$ is defined as
\[
w_\mathbf{v}(P) \ := \ \max \{\mathbf{v}^T\mathbf{x}\colon \mathbf{x}\in P\}-\min \{\mathbf{v}^T\mathbf{x}\colon \mathbf{x}\in P\} \, .
\]
The \textbf{lattice width} of $P$ is defined as $w(P) := \min \{w_\mathbf{v} (P) \colon
\mathbf{v}\in \mathbb{Z}^d\setminus \{ \bzero \}\}$. We note that the lattice width of a lattice polytope is always a nonnegative integer and attained in a primitive direction.

For any $\mathbf{v}\in \mathbb{R}^d\setminus \{ \bzero \}$ and $m\in \RR$, we denote
$H_\mathbf{v}(m)=\{\mathbf{x}\in \RR ^d\colon \mathbf{v}^T\mathbf{x}=m\}$. We call a
direction $\mathbf{v}\in \ZZ^d\setminus \{ \bzero \}$ \textbf{facet defining} for a
lattice polytope $P$ if $H_\mathbf{v}(m)\cap P$ is a facet of $P$ for some integer $m$.
The next lemma shows that in case of lattice zonotopes with lattice width $1$, the lattice width is always
attained by a facet-defining direction.
\begin{lemma}\label{lem:zonotopewidth1}
Let $Z$ be a $d$-dimensional lattice zonotope with lattice width $1$. Then there exists a facet
defining direction $\mathbf{v}\in \ZZ^d\setminus \{ \bzero \}$ with $w_\mathbf{v}(Z)=1$. In this case, $Z\cong Z'\times [0,1]$ where $Z'$ is a $(d-1)$-dimensional lattice zonotope.
\end{lemma}
\begin{proof}
Let $Z=Z(\mathbf{v}_1,\ldots, \mathbf{v}_k)$ be a $d$-dimensional lattice zonotope with
lattice width $1$ and let $\bv\in \ZZ^d\setminus \{ \bzero \}$ be a primitive direction with $w_\bv (Z)=1$. Let $F=\{\mathbf{p}\in Z\colon \mathbf{v}^T\mathbf{p}=\max _{\mathbf{q}\in Z}\mathbf{v}^T\mathbf{q} \}$ be the face of $Z$ that is maximized in direction $\mathbf{v}$. Then $F$ is of the form
\[
F \ = \ \left\{\sum _{i\in I_0} \lambda _i\bv _i \colon 0\leq \lambda_i \leq 1\right\} +\sum _{i\in I_+}\bv_i
\]
where 
\[
I_0 \ = \ \left\{i\in [d]\colon \bv^T \bv _i =0 \right\}
\qquad \text{ and } \qquad
I_\pm \ = \ \left\{i\in [d]\colon \pm \bv^T \bv _i >0 \right\}. 
\]
If $|I_+\cup I_-|\geq 2$ then for any $i, j \in I_+\cup I_-$, $i\neq j$, then we observe that the number of values $|\{0,\bv^T\bv_{i},\bv^T\bv_{j},\bv^T\bv_{i}+\bv^T\bv_{j})\}|$ is always at least $3$. Since the vectors $0,\bv_{i},\bv_j,\bv_{i}+\bv_{j}$ are contained in $Z$ we have $w_\bv (Z)\geq 2$, a contradiction.

This shows that $|I_+\cup I_-| =1$ and it follows that $\bv$ is a facet-defining direction. Further, if we
denote by $\bw$ the unique generator such that $\bv^T\bw\neq 0$ then $\bv^T\bw=\pm 1$ since $w_\bv (Z)=1$. Without loss of generality, let $\bv^T\bw=-1$, that is, $\bw \in I_-$. Then $Z=F+[\mathbf{0},\bw]$. (In the other case, $Z$ is a translate of that polytope by the integer vector $\bw$.) Since $w_\bv (Z)=1$ there are no lattice points in
between the affine hyperplanes $\aff F$ and $\aff F +\bw$ and thus any lattice basis
$\{\mathbf{b}_1,\ldots,\mathbf{b}_{d-1}\}$ of $\aff F \cap \ZZ^d$ can be complemented with
$\bw$ to a basis of $\ZZ^d$. Let $\varphi\colon \RR^d \rightarrow \RR^d$ be the affine lattice
preserving map defined by $\bw\mapsto \mathbf{e}_d$ and $\mathbf{b}_i\mapsto \mathbf{e}_i$
for $1\leq i\leq d-1$. Then $\varphi (Z)=\varphi(F)\times [0,1]$ and, thus, $Z$ is
unimodularly equivalent to the product of the $(d-1)$-dimensional lattice zonotope $\varphi (F)$ and the unit segment. This proves the claim.
\end{proof}

\begin{remark}
Lemma~\ref{lem:zonotopewidth1} does not carry over to arbitrary lattice polytopes or zonotopes of width larger than $1$.
For example, the simplex with vertices $(0,0,0), (1,1,0),(1,0,1)$ and $(0,1,1)$ has lattice width $1$, since it is contained in the unit cube $[0,1]^3$. However, its width in all four (primitive) facet directions is $2$. Furthermore, if we consider the lattice parallelepiped generated by the vectors $[1,0,0],[0,1,0]$ and $[1,1,m]$ we see that it is contained in the axes-parallel box $[0,2]^2\times [0,m]$. Except for its vertices its only other  lattice points are $[1,1,1],[1,1,2],\ldots, [1,1,m-1]$. In particular, the lattice width is $2$, but the width in all of its primitive facet directions is equal to $m$.
\end{remark}

We use the lemmas above to obtain geometric constraints on $3$-dimensional lattice zonotopes of degree $2$.

\begin{lemma} \label{width}
Every $3$-dimensional lattice zonotope of degree $2$ has lattice width $1$ or is a lattice parallelepiped.
\end{lemma}
\begin{proof}
Let $Z$ be a $3$-dimensional lattice zonotope of lattice width bigger than $1$ and assume that $Z$ is generated by pairwise linearly independent vectors $\bv_{1}, \dots, \bv_{k} \in \mathbb{Z}^{3}$ for some $k \geq 4$.
We will show that $Z$ contains a lattice point in its interior (and thus is of degree~3).

Since $Z$ is $3$-dimensional, there is a set of three generating vectors that forms a basis of $\mathbb{R}^{3}$, say $\{\bv_{1}, \bv_{2}, \bv_{3}\}$.
Thus, there exist $\alpha_{i} \in \mathbb{R}_{\geq 0}$ and $\sigma_{i} \in \{1, -1\}$ for all $i \in \{1, 2, 3\}$ such that $ \bv_{4} = \sigma_{1} \alpha_{1} \bv_{1} + \sigma_{2} \alpha_{2} \bv_{2} + \sigma_{3} \alpha_{3} \bv_{3}.$ Since $Z(\pm \bv_{1}, \dots, \pm \bv_{k} )$ (again fixing one sign per generator) is a translate of $Z$ by a vector in $\mathbb{Z}^d$ we may assume that $\sigma _1=\sigma _2=\sigma _3=1$. 

We distinguish two cases. First, we assume that $\alpha_{i} > 0$ for all $i \in \{1, 2, 3\}$. If we set
$$\epsilon \ := \ \frac{1}{2 \max \left\{1, \alpha_{i}: i \in \{1, 2, 3\}\right\}}$$ 
then we have $1 - \epsilon \in (0,1)$ and $\epsilon \alpha_{i} \in (0,1)$ for all $i \in \{1, 2, 3\},$ and we can write
\[
\bv_{4} \ = \ \alpha_{1} \bv_{1} + \alpha_{2} \bv_{2} + \alpha_{3} \bv_{3} \ = \ \epsilon
\alpha_{1} \bv_{1} + \epsilon \alpha_{2} \bv_{2} + \epsilon \alpha_{3} \bv_{3} + (1 -
\epsilon) \bv_{4} \, .
\]
Thus, $\bv_{4}$ lies in the interior of $Z(\bv_{1}, \bv_{2}, \bv_{3}, \bv_{4})\subseteq Z$.

In the other case, if there is an index $j \in \{1, 2, 3\}$ such that $\alpha_{j} = 0$,
then $\alpha_{i} > 0$ for all $i \in \{1, 2, 3\} \setminus \{j\}$ since the generating
vectors are pairwise linearly independent. Without loss of generality, let $\alpha_{1} =
0$. Then $\bv_{2}, \bv_{3}, \bv_{4}$ span a plane. Let $\mathbf{v}\in
\mathbb{Z}^3\setminus \{ \bzero \}$ be a normal vector to this plane, and let 
\[
  m \ := \ \min \left\{\mathbf{v}^T\mathbf{x}\colon \mathbf{x}\in Z \right\}
    \ <  \ \max \left\{\mathbf{v}^T\mathbf{x}\colon \mathbf{x}\in Z \right\}
    \ =: \ M \, . 
\]
Since $w_\mathbf{v}(Z)\geq 2$, there exists an integer $m<\ell <M$ such that $H_\mathbf{v}(\ell)\cap \ZZ^3\neq \emptyset$. Further, since $Z$ is a zonotope, $F:=Z\cap H_\mathbf{v}(m)$ and $\bar F:=Z\cap H_\mathbf{v}(M)$ are congruent facets of $Z$ that are integer translates of a zonotope generated by a superset of $\{\bv_2,\bv_3,\bv_4\}$. It follows that $H_\mathbf{v}(\ell)\cap \conv(F\cup \bar F)$ is a translate of $F$ (and $\bar F$) and thus contains an element of the affine lattice $H_\mathbf{v}(\ell)\cap \ZZ^3$ in its relative interior by Lemma~\ref{2dimintpoint}; here we are using our assumption that the generators are pairwise linearly independent. Since $H_\mathbf{v}(\ell)\cap \conv(F\cup \bar F)$ is not contained in a facet this lattice point lies in the interior of $Z$, and the proof is complete.
\end{proof}
Next, we focus on the geometry of lattice parallelepipeds of width larger than $1$.

\def\JPicScale{0.8}
\begin{figure}[htb]
\begin{center}
\input{width2color}
\end{center}
\caption{The parallelepiped (and its lattice points) in Lemma~\ref{lem:width2}.}\label{fig:width2}
\end{figure}
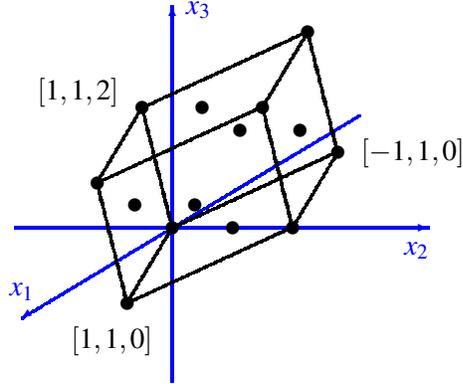

\begin{lemma}\label{lem:width2}
Every $3$-dimensional lattice parallelepiped of degree $2$, with lattice width greater
than $1$ and only primitive generators is unimodularly equivalent to the parallelepiped $Z(\mathbf{v}_1,\mathbf{v}_2,\mathbf{v}_3)$ where
\[
\mathbf{v}_1 = [1,1,0] \, , \ \mathbf{v}_2 = [-1,1,0] \, , \ \mathbf{v}_3 = [1,1,2] \, .
\]
\end{lemma}
\begin{proof}
Let $Z=Z(\bv_1,\bv_2,\bv_3)$ be a $3$-dimensional lattice parallelepiped of degree $2$,
lattice width greater than $1$ and primitive generators $\bv_1,\bv_2,\bv_3$. In order to
prove the claim we first show that every facet contains at most one lattice point besides
its vertices. Since $\bv_1,\bv_2,\bv_3$ are primitive, this lattice point has to lie in
the relative interior of the facet. To see that, without loss of generality we assume that the facet
$F:=Z(\bv_1,\bv_2)$ contains at least $2$ lattice points in its relative interior. Then
the solid-angle sum of $F$ with respect to the induced metric of its affine hull is at
least~$3$. Let $\bz\in \ZZ^3\setminus \{ \bzero \}$ be a normal vector of $F$, then
$F=Z\cap H_\bz (0)$ and $F+\bv_3=Z(\bv_1,\bv_2)+\bv_3=Z\cap H_\bz (M)$ for some $M\in \ZZ$. Since $Z$ has
width larger than $1$ there is an integer $\ell$ lying between $0$ and $M$ such that
$H_\bz(\ell)$ contains a lattice point. Since $F'=H_\bz(\ell)\cap Z$ is a translate of $F$
it also has solid-angle sum at least~$3$ by Lemma~\ref{lem:solidangleconstant}. Since $\bv_3$ is a primitive generator and $Z$
does not contain lattice points in its interior all lattice points of $F'$ must lie on the
relative interior of its edges. Every such lattice point contributes $\frac 1 2$ to the
solid-angle sum of $F'$ and therefore $F'$ contains at least $6$ lattice points. In particular,
there must be an edge of $F'$ that contains at least two lattice points in its relative
interior. This is a contradiction to $\bv_1$ and $\bv_2$ being primitive generators. We
thus conclude that every facet of $Z$ contains at most one lattice point in addition to its vertices. By central symmetry of the facets, any such lattice point must lie on the intersection of the diagonals of the facet.

Next, we show that $Z$ has exactly one lattice point in the relative interior of every
facet. Let us assume that there is a facet, say $F=Z(\bv_1,\bv_2)$, that does
not contain any lattice point except for its four vertices. Then we claim that not all of the four other
non-parallel facets contain an interior point. To see that, assume that this is the case.
Then, with the notation above, these four lattice points all lie on the hyperplane $H_\bz
(\ell)$ with $\ell=M/2$, and contribute each $\frac 1 2$ to the solid-angle sum of
$F'=H_\bz(\ell)\cap Z$ which therefore equals $2$, a contradiction since $F'$ is a
translate of $F$ which has solid-angle sum $1$. This shows that at most one pair of parallel facets contains a lattice point in each relative interior. If no facet contains a lattice point in
its relative interior, then $Z$ does not contain any lattice points other than its vertices and is
thus unimodularly equivalent to the unit cube which has width $1$, a contradiction. In the
other case, assume that $Z(\bv_2,\bv_3)$ does not contain an interior lattice point, but
$Z(\bv_1,\bv_3)$ contains a lattice point in its relative interior. Let $\bw\in
\ZZ^3\setminus \{ \bzero \}$ be a primitive, facet-defining direction for $Z(\bv_1,\bv_3)$. Then we
claim that $w_\bw (Z)=1$ --- otherwise, with analogous arguments as above, there is a slice of $Z$ parallel to $Z(\bv_1,\bv_3)$ but different from $Z(\bv_1,\bv_3)$ and $Z(\bv_1,\bv_3)+\bv_2$ with solid-angle sum $2$. This is a contradiction, since all lattice points in $Z$ are contained in $Z(\bv_1,\bv_3)$ and the parallel facet $Z(\bv_1,\bv_3)+\bv_2$. Therefore, $w_\bw (Z)=1$ and $Z$ has lattice width $1$, again a contradiction.

We therefore have proved that $Z$ contains precisely one lattice point at the intersection
of the diagonals of every facet. We claim that, up to unimodularly equivalence, there is only one such parallelepiped. For that we observe that the width in all three facet-defining directions is $2$. In particular, the lattice parallelepiped spanned by $\mathbf{b}_1=\bv_1, \mathbf{b}_2=\frac 1 2(\bv_1+\bv_2)$ and $\mathbf{b}_3=\frac 1 2(\bv_1+\bv_3)$ contains only its vertices as lattice points. Thus, its generators form a lattice basis of $\ZZ^3$. We conclude by observing that the map defined by $\mathbf{b}_1 \mapsto [1,1,0]$, $\mathbf{b}_2 \mapsto [0,1,0]$, $\mathbf{b}_3\mapsto [1,1,1]$ is unimodular, and maps $Z$ to the above lattice parallelepiped. This completes the proof.
\end{proof}

\begin{proof}[{Proof of Theorem~\ref{thm:degree2zonoclass}}]
Let $Z$ be a $3$-dimensional lattice zonotope of degree $2$. 

If $Z$ has lattice width $1$ then, by Lemma~\ref{lem:zonotopewidth1}, $Z$ is unimodularly equivalent to $Q\times [0,1]$ for some $2$-dimensional lattice zonotope $Q$. On the other hand, every $3$-dimensional lattice zonotope of this form has no interior points and is therefore of degree $2$. 

If $Z$ has width greater than $1$ then, by Lemma~\ref{width}, $Z$ is a parallelepiped. We will show that every $3$-dimensional lattice parallelepiped of degree $2$ and lattice width greater than $1$ has only primitive generators. This, together with Lemma~\ref{lem:width2} will then complete the proof.

Let $Z = Z (\bu, \bv, \bw)$ be a $3$-dimensional lattice parallelepiped of degree $2$ and lattice width greater than~$1$. 

First we show that $Z$ can have at most one non-primitive generator. For that we assume that both $\bu$ and $\bv$ are not primitive. Since $\bu$ is not primitive there exists a $0<\lambda <1$ such that $\lambda \bu \in \mathbb{Z}^3$. Let $\lambda_{\bu} := \min \{ \lambda \in (0,1): \, \lambda \bu \in \mathbb{Z}^{3}\}$ and
$\tilde{\bu} := \lambda_{\bu} \bu$; note that $\lambda_{\bu} \le \frac 1 2$. Let
$\lambda_{\bv}\leq \frac 1 2 $ and $\tilde{\bv}$ be defined analogously.
Consider the facets $F := Z (\bu, \bv)$ and $\overline{F} := Z (\bu, \bv) + \bw$ of $Z$,
and let $\mathbf{z}\in \ZZ^3\setminus \{ \bzero \}$ be an inner normal vector of $F$. Then
$\overline F$ is contained in $H_\mathbf{z}(M)$ for some integer $M>0$ and $F$ is contained in $H_\mathbf{z}(0)$. Since $w(Z) \geq 2$, there exists an integer $0<\ell < M$ such that $H_\mathbf{z}(\ell)\cap \ZZ^3\neq \emptyset$. We consider the translated parallelogram  
\[
Y \ := \ Z (\tilde{\bu}, \tilde{\bv}) + \frac{1}{2} (\tilde{\bu} + \tilde{\bv})
  \  = \  \left\{\lambda_{1} \tilde{\bu} + \lambda_{2} \tilde{\bv} : \, \frac{1}{2} \leq
\lambda_{i} \leq \frac{3}{2} \ \textup{ for } i  = 1,2 \right\} .
\]
We observe that $Y$ is contained in the relative interior of $F$ since $0\leq \lambda_{\bu},\lambda_{\bv}\leq \frac 1 2$. Thus, also $\overline{Y}=Y+\bw$ is contained in the relative interior of the parallel facet $\overline{F}$. The intersection $H_\mathbf{z} (\ell)\cap \conv \{Y\cup \overline{Y}\}$ is therefore a translate of the lattice parallelogram $Z (\tilde{\bu}, \tilde{\bv})$ which is contained in the interior of $Z$. By Lemma~\ref{lem:solidangleconstant}, this translate contains a lattice point in the affine lattice $H_\mathbf{z}(\ell)\cap \ZZ^3$, and thus, $Z$ contains a lattice point in the interior. This contradicts the assumption that $Z$ has degree $2$ and thus $Z$ can have at most one non-primitive generator.

In order to show that $Z= Z (\bu, \bv, \bw)$ cannot have a single non-primitive generator, we assume that $\bu$ and $\bv$ are primitive and that $\bw$ is not primitive, that is, $\mathcal{L}(\bw)>1$. Let $\lambda_{\bw} := \textup{min} \{ \lambda \in (0,1): \lambda \bw \in \mathbb{Z}^{3}\}$ and $\tilde{\bw} := \lambda_{\bw} \bw$.
Then $\tilde{\bw}$ is a primitive lattice vector and thus $Z(\bu,\bv,\tilde{\bw})$ has only primitive generators. Further, we observe that 
\[
Z \ = \bigcup_{i=0}^{\mathcal{L}(\bw)-1}  \left( Z(\bu,\bv,\tilde{\bw})+i\,\tilde{\bw}
\right) 
\]
is a ``stack'' of translates of $Z(\bu,\bv,\tilde{\bw})$. Since $Z$ is of degree $2$, $Z$ does not have any interior lattice points. Since $\mathcal{L}(\bw)>1$, $Z$ is the union of at least two copies of $Z(\bu,\bv,\tilde{\bw})$ that are attached at lattice translates of the facet $Z(\bu,\bv)$. Thus, $Z(\bu,\bv)$ cannot have interior lattice points. In order to arrive at a contradiction, we consider the lattice width of $Z(\bu,\bv,\tilde{\bw})$. If $Z(\bu,\bv,\tilde{\bw})$ has lattice width greater than $1$ then, by Lemma~\ref{lem:width2}, $Z(\bu,\bv,\tilde{\bw})$ is unimodulary equivalent to $Z(\mathbf{v}_1,\mathbf{v}_2,\mathbf{v}_3)$, where $\mathbf{v}_1 = [1,1,0]$, $\mathbf{v}_2 = [-1,1,0]$ and $\mathbf{v}_3 = [1,1,2]$. This polytope has a lattice point in the relative interior of every facet, a contradiction.  In the other case, if $Z(\bu,\bv,\tilde{\bw})$ has lattice width $1$, then by Lemma~\ref{lem:zonotopewidth1}, there exists a facet-defining direction $\bz \in \ZZ^3$ with $w_\bz (Z(\bu,\bv,\tilde{\bw}))=1$. We observe that since $Z$ is a ``stack'' of copies of $Z(\bu,\bv,\tilde{\bw})$, any direction that is orthogonal to the facets $Z(\bu,\tilde{\bw})$ or $Z(\bv,\tilde{\bw})$ is also facet defining for $Z$. Thus, if $\bz$ is orthogonal to the facets $Z(\bu,\tilde{\bw})$ or $Z(\bv,\tilde{\bw})$ then the lattice width of $Z$ is equal to $1$, a contradiction. In the other case, if $\bz$ is orthogonal to $Z(\bu,\bv)$, we observe that $Z(\bu,\bv,\tilde{\bw})$ does not contain any other lattice points except for its vertices. It follows that $\{\bu,\bv,\tilde{\bw}\}$ is a lattice basis for $\ZZ^3$ and $Z$ is unimodularly equivalent to $[0,1]^2\times [0,\mathcal{L}(\bw)]$. Thus, again it follows that $Z$ has lattice width $1$ which contradicts the assumptions. This shows that $Z$ can have only primitive generators. This completes the proof.
\end{proof}

\begin{proof}[{Proof of Theorem~\ref{thm:degree2zono}}]
By Theorem~\ref{thm:degree2zonoclass}, a $3$-dimensional lattice zonotope $Z$ has degree $2$ if and only if it is unimodularly equivalent to either $Z(\mathbf{v}_1,\mathbf{v}_2,\mathbf{v}_3)$, where $\mathbf{v}_1 = [1,1,0]$, $\mathbf{v}_2 = [-1,1,0]$ and $\mathbf{v}_3 = [1,1,2]$, or $Q\times [0,1]$ where $Q$ is a $2$-dimensional lattice zonotope. 

Let $Q$ be a $2$-dimensional lattice zonotope with Ehrhart polynomial $\ehr _Q(n)=(n+1)^2+c_1(n+1)n+c_2n^2$ and let $Z\cong Q\times [0,1]$. Then
\[
\ehr _{Z}(n)
\ = \ \left((n+1)^2+c_1(n+1)n+c_2n^2\right)(n+1)
\ = \ (n+1)^3+c_1(n+1)^2n+c_2(n+1)n^2 .
\]
In particular, a polynomial $(n+1)^3+c_1(n+1)^2n+c_2(n+1)n^2+c_3n^3$ is the Ehrhart polynomial of a lattice polytope of the form $Q\times [0,1]$ if and only if $c_1,c_2,c_3$ are integers, $c_3=0$ and, by Theorem~\ref{thm:2imzono}, $c_2=0$ or $c_2\geq c_1-1$. These are precisely the claimed conditions (i) and (ii).

Finally, by Theorem~\ref{thm:Stanley}, the Ehrhart polynomial of
$Z(\mathbf{v}_1,\mathbf{v}_2,\mathbf{v}_3)$ equals $(n+1)^3+3(n+1)n^2$. In particular, the Ehrhart polynomial of $Z(\mathbf{v}_1,\mathbf{v}_2,\mathbf{v}_3)$ satisfies the conditions (i) and (ii). This completes the proof.
\end{proof}

\section{$h^\ast$-polynomials}\label{sec:bases}
In~\cite{beulerian} a formula for the $h^\ast$-polynomial of lattice
zonotopes in terms of refined Eulerian polynomials was given. Let $S_d$ denote the set of all permutations on
$[d]=\{1,\ldots, d\}$. For every permutation $\sigma\in S_d$, we call $i<d$ a
\textbf{descent} of $\sigma$ if $\sigma (i)>\sigma (i+1)$. The number of descents of
$\sigma$ is denoted by $\des \sigma$. The \Def{Eulerian polynomial} is defined as $A^d (t) \
= \ \sum _{\sigma \in S_d} t^{\des (\sigma)}$. For all $1\leq j \leq d$, the
\Def{$(A,j)$-Eulerian polynomial} is defined as
\[
A^d _j  (t) \ = \sum _{\sigma \in S_d\atop \sigma (d) = d+1-j} t^{\des (\sigma)} \, .
\]
For example, 
\begin{align*}
A_1^3 (t) &= 1+t, & A_1^4 (t) &= 1+4t+t^2, \\
A_2^3 (t) &= 2t, & A_2^4 (t) &= 4t+2t^2, \\
A_3^3 (t) &= t+t^2, & A_3^4 (t) &= 2t+4t^2, \\
&& A_4^4 (t) &= t+4t^2+t^3.
\end{align*}
In~\cite{beulerian} it was shown that the $h^\ast$-polynomial of any lattice zonotope is a
nonnegative linear combination of $(A,j)$-Eulerian polynomials. The following is a summary of~\cite[Theorem 1.4, Theorem 4.2, Corollary 4.4]{beulerian}.

\begin{theorem}[Beck--Jochemko--McCullough \cite{beulerian}]\label{thm:hstarzonotope}
Let $\ehr _Z (n)=(n+1)^d+c_1 (n+1)^{d-1}n+\cdots +c_d n^d$ be the Ehrhart polynomial of a $d$-dimensional lattice zonotope $Z$. Then $c_1,\ldots, c_d$ are nonnegative integers and 
\[
h^\ast _Z (t) \ = \ A^{d+1} _1 (t) + c_1 \, A^{d+1} _2 (t) + \cdots + c_d \, A^{d+1}_{d+1} (t) \, .
\]
\end{theorem}
Theorem~\ref{thm:hstarzonotope} together with Theorems~\ref{thm:2imzono} and~\ref{thm:degree2zono} allow us to give a complete classification of all $h^\ast$-polynomials of lattice zonotopes of degree~$2$.

\begin{theorem}\label{thm:hstar2dim}
The polynomial
\[
1+h^\ast _1 t+h^\ast _2 t^2
\]
is the $h^\ast$-polynomial of a $2$-dimensional lattice zonotope if and only if
\begin{itemize}
\item[(i)] $h_1^\ast,h_2^\ast\in \ZZ_{\geq 0}$ and $h_1^\ast -h_2^\ast \equiv 1\bmod 2$ and
\item[(ii)] $h_2^\ast +1\leq h_1^\ast$ and
\item[(ii)] $h_1^\ast \leq 3h_2^\ast +3$ or $h_2^\ast =0$.
\end{itemize}
\end{theorem}
\begin{proof}
From Theorem~\ref{thm:hstarzonotope} applied to the case $d=2$ we obtain that the
injective function $(c_1,c_2)\mapsto (h_1^\ast,h_2^\ast):=(1+2c_1+c_2,c_2)$ maps the coefficients of an Ehrhart polynomial $(n+1)^2+c_1(n+1)n+c_2n^2$ of a $2$-dimensional lattice zonotope to the coefficients of its $h^\ast$-polynomial. The inverse map is given by
\begin{eqnarray*}
c_1 &=& \frac{h_1^\ast -h_2^\ast -1}{2}\\
c_2&=&h_2^\ast \, .
\end{eqnarray*}
In particular, by Theorem~\ref{thm:2imzono}, the polynomial $1+h^\ast _1 t+h^\ast _2 t^2$, $h_1^\ast ,h_2^\ast \in \ZZ_{\geq 0}$, is the $h^\ast$-polynomial of a $2$-dimensional lattice zonotope if and only if (i) $\frac 1 2(h_1^\ast -h_2^\ast -1)$ is an integer which is equivalent to $h_1^\ast -h_2^\ast \equiv 1\bmod 2$ and nonnegative which is equivalent to $h_2^\ast +1\leq h_1^\ast$; and (ii) $h_2^\ast =0$ or $h_2\geq \frac 1 2(h_1^\ast -h_2^\ast -1)-1$ which is equivalent to $h_1^\ast \leq 3h_2^\ast +3$.
\end{proof}

\begin{theorem}\label{thm:hstardeg2}
The polynomial
\[
1+h^\ast _1 t+h^\ast _2 t^2
\]
is the $h^\ast$-polynomial of a $3$-dimensional lattice zonotope of degree $2$ if and only if
\begin{itemize}
\item[(i)] $h_1^\ast,h_2^\ast\in \ZZ_{\geq 0}$, $2h_1^\ast -h_2^\ast \equiv 1\bmod 6$ and $2h_2^\ast -h_1^\ast \equiv 4\bmod 6$ and
\item[(ii)] $\frac 1 2h_1^\ast-1\leq h_2^\ast \leq 2h_1^\ast-7$ and 
\item[(iii)] $h_2^\ast \geq h_1^\ast -5$ or $2h_2^\ast=h_1^\ast -2$.
\end{itemize}
\end{theorem}
\begin{proof}
From Theorem~\ref{thm:hstarzonotope} we deduce that $(c_1,c_2)\mapsto (h_1^\ast,h_2^\ast):=(4+4c_1+2c_2,1+2c_1+4c_2)$ maps coefficients of Ehrhart polynomials of $3$-dimensional lattice zonotopes of degree $2$ to the coefficients of their $h^\ast$-polynomials. Equivalently,
\begin{eqnarray*}
c_1 &=& \frac{2h_1^\ast -h_2^\ast -7}{6}\\
c_2&=& \frac{2h_2^\ast -h_1^\ast +2}{6}\, .
\end{eqnarray*}
Thus, by Theorem~\ref{thm:degree2zono}, $1+h^\ast _1 t+h^\ast _2 t^2$, $h_1^\ast,h_2^\ast \in \ZZ_{\geq 0}$, is the $h^\ast$-polynomial of a $3$-dimensional lattice zonotope of degree $2$ if and only if (i) $2h_1^\ast -h_2^\ast -7$ and $2h_2^\ast -h_1^\ast +2$ are divisible by $6$ and greater or equal to $0$. This is equivalent to $2h_1^\ast -h_2^\ast \equiv 1\bmod 6$ and $2h_2^\ast -h_1^\ast \equiv 4\bmod 6$ and $\frac 1 2h_1^\ast-1\leq h_2^\ast \leq 2h_1^\ast-7$; and (ii) $2h_2^\ast -h_1^\ast +2=0$ or $2h_2^\ast -h_1^\ast +2\geq 2h_1^\ast -h_2^\ast -7-6$ which is equivalent to $h_2^\ast \geq h_1^\ast -5$.
\end{proof}

\textbf{Acknowledgements:} We would like to thank the anonymous referee for helpful comments. KJ was partially supported by the Wallenberg AI, Autonomous Systems and Software Program funded by
the Knut and Alice Wallenberg Foundation, as well as Swedish Research Council grant
2018-03968 and the G\"oran Gustafsson Foundation.

\nocite{janssen}
\bibliographystyle{amsplain}
\bibliography{bib}

\setlength{\parskip}{0cm} 

\end{document}

%% file: width2color
\ifx\JPicScale\undefined\def\JPicScale{1}\fi
\unitlength \JPicScale mm
\begin{picture}(72.5,69.5)(0,0)
\color{blue}
\linethickness{0.3mm}
\put(3.75,25.62){\line(1,0){68.75}}
\put(72.5,25.62){\vector(1,0){0.12}}
\linethickness{0.3mm}
\put(30,0){\line(0,1){62.5}}
\put(30,62.5){\vector(0,1){0.12}}
\linethickness{0.3mm}
\multiput(5,10.62)(0.2,0.12){281}{\line(1,0){0.2}}
\put(5,10.62){\vector(-3,-2){0.12}}
\color{black}
\linethickness{0.3mm}
\multiput(25,45.62)(0.12,-0.48){42}{\line(0,-1){0.48}}
\linethickness{0.3mm}
\multiput(22.5,13.12)(0.12,0.2){63}{\line(0,1){0.2}}
\linethickness{0.3mm}
\multiput(22.5,13.12)(0.26,0.12){104}{\line(1,0){0.26}}
\linethickness{0.3mm}
\multiput(50,25.62)(0.12,0.2){63}{\line(0,1){0.2}}
\linethickness{0.3mm}
\multiput(17.5,33.12)(0.12,0.2){63}{\line(0,1){0.2}}
\linethickness{0.3mm}
\multiput(17.5,33.12)(0.12,-0.48){42}{\line(0,-1){0.48}}
\linethickness{0.3mm}
\multiput(45,45.62)(0.12,0.2){63}{\line(0,1){0.2}}
\linethickness{0.3mm}
\multiput(45,45.62)(0.12,-0.48){42}{\line(0,-1){0.48}}
\linethickness{0.3mm}
\multiput(52.5,58.12)(0.12,-0.48){42}{\line(0,-1){0.48}}
\color{blue}
\put(70.62,21.88){\makebox(0,0)[cc]{$x_2$}}
\put(5,15){\makebox(0,0)[cc]{$x_1$}}
\put(34.38,61.88){\makebox(0,0)[cc]{$x_3$}}
\color{black}
\linethickness{0.3mm}
\multiput(30,25.62)(0.26,0.12){104}{\line(1,0){0.26}}
\linethickness{0.3mm}
\multiput(17.5,33.12)(0.26,0.12){104}{\line(1,0){0.26}}
\linethickness{0.3mm}
\multiput(25,45.62)(0.26,0.12){104}{\line(1,0){0.26}}
\put(70,38.12){\makebox(0,0)[cc]{$[-1,1,0]$}}

\put(20,6.88){\makebox(0,0)[cc]{$[1,1,0]$}}

\put(14.38,48.12){\makebox(0,0)[cc]{$[1,1,2]$}}

\linethickness{0.3mm}
\put(23.75,29.38){\circle*{2.03}}

\linethickness{0.3mm}
\put(35,45.62){\circle*{2.03}}

\linethickness{0.3mm}
\put(41.25,41.88){\circle*{2.03}}

\linethickness{0.3mm}
\put(45,45.62){\circle*{2.03}}

\linethickness{0.3mm}
\put(25,45.62){\circle*{2.03}}

\linethickness{0.3mm}
\put(52.5,58.12){\circle*{2.03}}

\linethickness{0.3mm}
\put(57.5,38.12){\circle*{2.03}}

\linethickness{0.3mm}
\put(50,25.62){\circle*{2.03}}

\linethickness{0.3mm}
\put(17.5,33.12){\circle*{2.03}}

\linethickness{0.3mm}
\put(22.5,13.12){\circle*{2.03}}

\linethickness{0.3mm}
\put(30,25.62){\circle*{2.03}}

\linethickness{0.3mm}
\put(51.25,41.88){\circle*{2.03}}

\linethickness{0.3mm}
\put(33.75,29.38){\circle*{2.03}}

\linethickness{0.3mm}
\put(40,25.62){\circle*{2.03}}

\end{picture}